\documentclass[11pt]{amsart}
\usepackage{amssymb,amsmath,epsfig,mathrsfs, enumerate, xparse, mathtools}
\usepackage[pagewise]{lineno}
\usepackage[nodisplayskipstretch]{setspace}
\usepackage{graphicx}\usepackage[normalem]{ulem}
\usepackage{fancyhdr}
\usepackage[margin=2.5cm]{geometry}
\usepackage{hyperref}
\hypersetup{
 colorlinks   = true,
 urlcolor     = blue,
 linkcolor    = blue,
 citecolor   = red ,
 bookmarksopen=true
}

\theoremstyle{plain}
\newtheorem{thrm}{Theorem}[section]
\newtheorem{lemma}[thrm]{Lemma}
\newtheorem{prop}[thrm]{Proposition}
\newtheorem{cor}[thrm]{Corollary}

\allowdisplaybreaks
\begin{document}
\newcommand{\sn}{\mathbb{S}^{n-1}}
\newcommand{\SL}{\mathcal L^{1,p}( D)}
\newcommand{\Lp}{L^p( Dega)}
\newcommand{\py}{  \partial_{x_{n+1}}^a}
\newcommand{\La}{\mathscr{L}_a}
\newcommand{\CO}{C^\infty_0( \Omega)}
\newcommand{\Rn}{\mathbb R^n}
\newcommand{\Rm}{\mathbb R^m}
\newcommand{\R}{\mathbb R}
\newcommand{\Om}{\Omega}
\newcommand{\Hn}{\mathbb H^n}
\newcommand{\aB}{\alpha B}
\newcommand{\eps}{\ve}
\newcommand{\BVX}{BV_X(\Omega)}
\newcommand{\p}{\partial}
\newcommand{\IO}{\int_\Omega}
\newcommand{\bG}{\boldsymbol{G}}
\newcommand{\bg}{\mathfrak g}
\newcommand{\bz}{\mathfrak z}
\newcommand{\bv}{\mathfrak v}
\newcommand{\Bux}{\mbox{Box}}
\newcommand{\e}{\ve}
\newcommand{\X}{\mathcal X}
\newcommand{\Y}{\mathcal Y}
\newcommand{\W}{\mathcal W}
\newcommand{\la}{\lambda}
\newcommand{\vf}{\varphi}
\newcommand{\rhh}{|\nabla_H \rho|}
\newcommand{\Ba}{\mathcal{B}_\beta}
\newcommand{\Za}{Z_\beta}
\newcommand{\ra}{\rho_\beta}
\newcommand{\n}{\nabla}
\newcommand{\vt}{\vartheta}
\newcommand{\its}{\int_{\{y=0\}}}

\numberwithin{equation}{section}

\newcommand{\RN} {\mathbb{R}^N}
\newcommand{\Sob}{S^{1,p}(\Omega)}
\newcommand{\Dxk}{\frac{\partial}{\partial x_k}}
\newcommand{\Co}{C^\infty_0(\Omega)}
\newcommand{\Je}{J_\ve}
\newcommand{\beq}{\begin{equation}}
\newcommand{\bea}[1]{\begin{array}{#1} }
\newcommand{\eeq}{ \end{equation}}
\newcommand{\ea}{ \end{array}}
\newcommand{\eh}{\ve h}
\newcommand{\Dxi}{\frac{\partial}{\partial x_{i}}}
\newcommand{\Dyi}{\frac{\partial}{\partial y_{i}}}
\newcommand{\Dt}{\frac{\partial}{\partial t}}
\newcommand{\aBa}{(\alpha+1)B}
\newcommand{\GF}{\psi^{1+\frac{1}{2\alpha}}}
\newcommand{\GS}{\psi^{\frac12}}
\newcommand{\HFF}{\frac{\psi}{\rho}}
\newcommand{\HSS}{\frac{\psi}{\rho}}
\newcommand{\HFS}{\rho\psi^{\frac12-\frac{1}{2\alpha}}}
\newcommand{\HSF}{\frac{\psi^{\frac32+\frac{1}{2\alpha}}}{\rho}}
\newcommand{\AF}{\rho}
\newcommand{\AR}{\rho{\psi}^{\frac{1}{2}+\frac{1}{2\alpha}}}
\newcommand{\PF}{\alpha\frac{\psi}{|x|}}
\newcommand{\PS}{\alpha\frac{\psi}{\rho}}
\newcommand{\ds}{\displaystyle}
\newcommand{\Zt}{{\mathcal Z}^{t}}
\newcommand{\XPSI}{2\alpha\psi \begin{pmatrix} \frac{x}{\left< x \right>^2}\\ 0 \end{pmatrix} - 2\alpha\frac{{\psi}^2}{\rho^2}\begin{pmatrix} x \\ (\alpha +1)|x|^{-\alpha}y \end{pmatrix}}
\newcommand{\Z}{ \begin{pmatrix} x \\ (\alpha + 1)|x|^{-\alpha}y \end{pmatrix} }
\newcommand{\ZZ}{ \begin{pmatrix} xx^{t} & (\alpha + 1)|x|^{-\alpha}x y^{t}\\
     (\alpha + 1)|x|^{-\alpha}x^{t} y &   (\alpha + 1)^2  |x|^{-2\alpha}yy^{t}\end{pmatrix}}
\newcommand{\norm}[1]{\lVert#1 \rVert}
\newcommand{\ve}{\varepsilon}
\newcommand{\D}{\operatorname{div}}
\newcommand{\G}{\mathscr{G}}

\title[On the space-like analyticity, etc.]{On the space-like analyticity in the extension problem for nonlocal parabolic equations}

\author{Agnid Banerjee}
\address{Tata Institute of Fundamental Research\\
Centre For Applicable Mathematics \\ Bangalore-560065, India}\email[Agnid Banerjee]{agnidban@gmail.com}

\author{Nicola Garofalo}
\address{University of Padova, Italy}\email[Nicola Garofalo]{nicola.garofalo@unipd.it}

\thanks{A. Banerjee  is supported in part by SERB Matrix grant MTR/2018/000267 and by Department of Atomic Energy,  Government of India, under
project no.  12-R \& D-TFR-5.01-0520. N. Garofalo is supported in part by a Progetto SID: ``Non-local Sobolev and isoperimetric inequalities", University of Padova, 2019.}

%
%
%

\maketitle

\begin{abstract}
In this note we give an elementary proof of the space-like real analyticity of solutions to a degenerate evolution problem that arises in the study of fractional parabolic operators of the type $(\p_t - \D_x(B(x)\nabla_x))^s$, $0<s<1$. Our primary interest is in the so-called \emph{extension variable}. We show that weak solutions that are even in such variable, are in fact real-analytic in the totality of the space variables. As an application of this result we prove the weak unique continuation property for nonlocal parabolic operators of the type above, where $B(x)$ is a uniformly elliptic matrix-valued function with real-analytic entries. 
\end{abstract}

\section{Introduction and Statement of the main result}

We consider the space $\R^{n+1}$ with generic variable $X=(x, x_{n+1})$, where $x\in \Rn$, $x_{n+1} \in \R$. We let $|X| = \sqrt{|x|^2 + x_{n+1}^2}$ and denote by $\mathbb B_r = \{X\in \R^{n+1}\mid |X|<r\}$ the open ball of radius $r$ centred at the origin. We also let $\R^{n+1}_+ =\{X\in \R^{n+1}\mid x_{n+1}>0\}$ and indicate the upper part of the open ball with $\mathbb B_r^{+} = \mathbb B_r \cap \R^{n+1}_+$. The symbol $\mathbb B_r^-$ will indicate the corresponding lower part of $\mathbb B_r$.  The thin space $\Rn\times\{0\}\subset \R^{n+1}$ will be routinely identified with $\R^n$, and we let $B_r = \mathbb B_r \cap \Rn$. 
We indicate with $x\to A(x) = [a_{ij}(x)]$ a given $(n+1)\times (n+1)$ matrix-valued function of the form
\begin{equation}\label{type}
a_{ij}(x) = \sum_{i,j=1}^n b_{ij} (x) e_i \otimes e_j  + e_{n+1} \otimes e_{n+1},
\end{equation} 
where the $b_{ij}$'s are assumed symmetric, uniformly elliptic, independent of $x_{n+1}$, and globally  real-analytic in $\R^n$. Moreover, we assume that the heat kernel of the parabolic operator $\partial_t - \D_x(B(x) \nabla_x)$  satisfies the stochastic completeness assumption \eqref{stoc} below. From now on, we indicate with $\D$ and $\nabla$ respectively the divergence and gradient with respect to the variable $X\in \R^{n+1}$. Given a function $U(X,t)$ in  $\R^{n+1}_+ \times \R$, and a number $a\in (-1,1)$, we refer to \eqref{nder} below for the meaning of weighted normal derivative $\py U((x,0),t)$ on the thin space $\Rn\times \R$.
The purpose of this note is to present an elementary proof of the following result. 

\begin{thrm}\label{main}
Let $a \in (-1, 1)$ and $T>0$. Assume that the function $U = U(X,t)$ is  a weak solution to the following degenerate parabolic problem 
\begin{equation}\label{p0}
\begin{cases}
\La U \overset{def} =\D(x_{n+1}^a A(x) \nabla U) - x_{n+1}^a  U_t=0\hspace{2mm} &\text{in}  \hspace{2mm}\mathbb B_3^+ \times (-1,T],
\\
\py U((x,0),t)= 0\hspace{2mm} &\text{in}  \hspace{2mm}B_3 \times (-1,T].
\end{cases}
\end{equation}
Let $\tilde U$ denote the even reflection of $U$ across $\{x_{n+1}=0\}$. Then $\tilde U\in C^\infty(\mathbb B_1 \times (0, T])$ and, for every fixed $t\in (0,T)$, the function  $X\to \tilde U(X,t)$ is real-analytic in $\mathbb B_1$. 
\end{thrm}

We note that if we replace the Neumann condition in \eqref{p0} with the Dirichlet assumption $U((x,0),t)=0$ in $B_3 \times (-1,T]$, then Theorem \ref{main} ceases to hold. Take for instance $B(x) \equiv \mathbb I_n$ and $U(X, t)= x_{n+1}^{1-a}$.
As it is well-known by now, if $s \in (0,1)$ and $a=1-2s$, then \eqref{p0} represents the parabolic counterpart of the Caffarelli-Silvestre  extension problem for the fractional operator $(\partial_t- \D_x(B(x) \nabla_x))^s$. In view of this aspect, as a consequence of Theorem \ref{main} we obtain the following.

\begin{cor}
Let $u(x,t)$ solve $(\partial_t -\D_x( B(x) \nabla_x))^s u=0$ in $B_1 \times (-1,0)$. Then $x\to u(x,t)$ is real-analytic in $B_1$ for any fixed $t$.
\end{cor}

The reader should note that we have assumed global real-analyticity of $B(x)$ only for the simplicity of exposition. We only need $B(x)$ to be real-analytic in, say,  $B_3$. Moreover, the stochastic completeness hypothesis \eqref{stoc} is easily ensured  by e.g. smoothly extending $B(x)$ to the whole of $\Rn$ in such a way that $B(x) \equiv  \mathbb I_n$ outside $B_5$. 

Real-analyticity results for nonlocal operators have a
long history that goes back to the seminal work of M. Riesz \cite{Ri} for the fractional Laplacian, and especially Kotake-Narasimhan \cite{KN} for operators of the type $L^s$, where $L$ is a general elliptic operator of order $2m$.
In the special case $A(x)\equiv \mathbb I_{n+1}$, and for time-independent solutions to \eqref{p0}, the real-analyticity has been recently proved in \cite[Appendix B]{JP} using functional analytic tools and weak type estimates. Differently from \cite{JP}, our approach is based on the fundamental solution of the extended operator in \eqref{p0} computed in \cite{Gcm}, which we use to obtain
an explicit Green type representation of symmetric solutions. In order to circumvent a singularity in such representation, we use: (a) the space-like real analyticity of the fundamental solution to complexify the space variables; (b) a limiting argument. Part (b) involves careful estimates of the heat kernel in \eqref{funda} below of the extended operator in \eqref{p0}, and this is the key novelty of our work. This general scheme, based on Green representation and  complexifying the space variables, is  inspired to an idea in \cite{John}  for the heat equation.

As  an interesting application of Theorem \ref{main} we obtain the following weak unique continuation property.

\begin{prop}\label{wucp}
Let $B(x)$ be  a uniformly elliptic matrix-valued function with real-analytic entries and let $u$ be a solution to $(\p_t - \D_x(B(x)\nabla_x))^s u=0$, such that $u=0$ in $B_1 \times (-1,0)$. Then $u \equiv 0$ in $\Rn \times (-1,0)$. 
\end{prop}

The reader should note that the previous result displays a purely nonlocal phenomenon: the zero set of $u$ propagates even in the region where the equation is not satisfied. Besides its own interest, Proposition \ref{wucp} finds application to Runge type approximations for inverse problems, and in the special case when $B(x) \equiv \mathbb I_n$ it has been earlier obtained in \cite[Prop. 5.5]{LLR} by means of Carleman estimates. Our elementary approach shows that the use of Carleman estimates can be avoided.

The paper is organized as follows. In Section \ref{s:n} we introduce some basic notations  and gather some preliminary results that are relevant to our work. In Section \ref{s:m} we prove Theorem \ref{main} and Proposition \ref{wucp}. 


\section{Preliminaries}\label{s:n}

In this section we collect some known results that will be used in this note. Without loss of generality, we  will assume that the uniformly parabolic operator $\partial_t - \D(B(x)\nabla_x)$ in $\Rn\times \R$ has a globally defined fundamental $p(x,x',t)$ that satisfies for every $x\in \Rn$ and $t>0$
\begin{equation}\label{stoc}
P_t 1(x,t) = \int_{\R^n} p(x,x',t) dx' = 1.
\end{equation}
{\allowdisplaybreaks
We will also assume without any restriction that 
the power series expansion centred at $0$ of the matrix $A(x)$ converges uniformly in $ B_8$. Henceforth, following a standard use we will write $z = x+iy, z' = x'+i y'$, etc. for points in $\mathbb C^n$. We will routinely identify with $x\in \Rn$ the point $x+ i 0\in \mathbb C^n$. In Section \ref{s:m} we will also need to complexify the thick space $\R^{n+1}$. Thus, we will denote by $Z = X+ i Y, Z' = X' + i Y'$, etc., points in $\mathbb C^{n+1}$. Again, the point $X= (x,x_{n+1})\in  \R^{n+1}$ will be routinely identified with $X + i 0\in \mathbb C^{n+1}$. We then have the following result, which is \cite[Theor. 8.1, p. 178]{Ei}. We state it as a lemma since in this note we make multiple references to it. 
} 

{\allowdisplaybreaks
\begin{lemma}\label{ei}
Let $x' \in  B_3$. For  $t>0$ the fundamental solution $p(x,x',t)$  can be analytically continued to a function $p(z,x',t)$ defined in the complex ball $\{z= x + i y  \in \mathbb C^n\mid |z-x'| < 5\}$. Moreover, there exist positive constants $C_0, c_0$ and $c_1$, depending on the ellipticity and the real-analytic character of the coefficient matrix $A(x)$, such that the following estimates hold 
\begin{equation}\label{gest}
\begin{cases}
|p(z,x',t)| \leq C_0 t^{-\frac{n}{2}} e^{-\frac{c_0 |x-x'|^2 -c_1|y|^2}{4t}},
\\
|\nabla_{x'} p(z,x',t) |\leq C_0 t^{-\frac{n+1}{2}} e^{-\frac{c_0 |x-x'|^2 -c_1|y|^2}{4t}}.
\end{cases} 
\end{equation} 
\end{lemma}
}
{\allowdisplaybreaks
With the notation in the opening of the note, given a function $U(X,t) = U((x,x_{n+1}),t)$ in $\R^{n+1}_+ \times (0,\infty)$, and a number $a\in (-1,1)$, we denote with $\py U$ the weighted normal derivative
\begin{equation}\label{nder}
\py U((x,0),t)\overset{def}{=}   \operatorname{lim}_{x_{n+1} \to 0^+}  x_{n+1}^a \partial_{x_{n+1}} U((x,x_{n+1}),t).
\end{equation}}
{\allowdisplaybreaks The second-order degenerate parabolic differential operator $\La$ is defined as in \eqref{p0} above. As we have mentioned in the introduction, if $s \in (0,1)$ and $a=1-2s$, then \eqref{p0} represents the parabolic counterpart of the Caffarelli-Silvestre  extension problem for the fractional operator $(\partial_t- \D(B(x) \nabla_x))^s$. For some background on this problem the reader is referred to \cite{Jo, CS, NS, ST, BG, BSt, LN}. 
}

{\allowdisplaybreaks We next recall that it was proved in \cite{Gcm} that, given $\phi \in C_0^{\infty}(\R^{n+1}_+)$, the solution of the Cauchy problem with Neumann condition
\begin{align}\label{CN}
\begin{cases}
\La U=0 \hspace{2mm} &\text{in}  \hspace{2mm}\R^{n+1}_+ \times (0,\infty)\\
U(X,0)=\phi(X), \hspace{2mm} &X \in \R^{n+1}_+,\\
\py U(x, 0, t) = 0 \hspace{2mm} & x\in \Rn,\ t \in (0, \infty)
\end{cases}
\end{align}
is given by the formula 
\begin{equation}\label{Ga}
\mathscr P^{(a)}_t \phi(X) \overset{def}{=} U(X,t) = \int_{\R^{n+1}_+} \phi(X') \G (X,X',t) (x'_{n+1})^a dX',
\end{equation}
where 
\begin{equation}\label{funda}
\G(X,X',t) = p(x,x',t)\ p^{(a)}(x_{n+1},x'_{n+1},t).
\end{equation}
In \eqref{funda} we have indicated with $p(x,x',t)$ the fundamental solution of the operator $\partial_t- \D(B(x) \nabla_x)$, and with $p^{(a)}(x_{n+1},x'_{n+1},t)$ that of the Bessel operator $\mathscr B_a = \p^2_{x_{n+1}} + \frac a{x_{n+1}} \p_{x_{n+1}}$ on $(\R^+,x_{n+1}^a dx_{n+1})$  with Neumann boundary condition in $x_{n+1}=0$ (reflected Brownian motion). Such function is given by the formula 
\begin{align}\label{fs}
p^{(a)}(x_{n+1},x'_{n+1},t) & =(2t)^{-\frac{a+1}{2}}\left(\frac{x_{n+1} x'_{n+1}}{2t}\right)^{\frac{1-a}{2}}I_{\frac{a-1}{2}}\left(\frac{x_{n+1} x'_{n+1}}{2t}\right)e^{-\frac{x_{n+1}^2+(x'_{n+1})^2}{4t}}.
\end{align}
}
In \eqref{fs} we have denoted by $I_{\frac{a-1}{2}}$ the modified Bessel function of the first kind and order $\frac{a-1}{2}$ defined by the power series 
\begin{align}\label{besseries}
 I_{\frac{a-1}{2}}(w) = \sum_{k=0}^{\infty}\frac{(w/2)^{\frac{a-1}{2}+2k} }{\Gamma (k+1) \Gamma(k+1+(a-1)/2)}, \hspace{4mm} |w| < \infty,\; |\operatorname{arg} w| < \pi.
\end{align}
From the asymptotic behaviour of $I_{\frac{a-1}{2}}(w)$ near $w=0$ and at infinity, one immediately obtains the following estimate for some $C(a), c(a) >0$ (see e.g. \cite[formulas (5.7.1) and (5.11.8)]{Le}) ,
\begin{equation}\label{bessel}
|I_{\frac{a-1}{2}}(w)| \leq C(a) |w|^{\frac{a-1}{2}}  \hspace{6mm} \text{if} \hspace{2mm} 0 < |w| \le c(a),\ \ \ \ \ I_{\frac{a-1}{2}}(w) \leq C(a) w^{-1/2} e^w \hspace{2mm}\ \  \text{if} \hspace{2mm} w \ge c(a).
\end{equation}

For future use we note explicitly that \eqref{funda} and \eqref{fs} imply that for every $x, x'\in \Rn$ and $t>0$ one has
\begin{equation}\label{neu}
\operatorname{lim}_{x_{n+1} \to 0^+}  x^a_{n+1} \partial_{x_{n+1}} \G((x,x_{n+1}),(x',0),t) =0. 
\end{equation}
We emphasise that for every fixed $t>0$ and $X'\in \R^{n+1}$ the function $X\to \G(X,X',t)$ is real-analytic in $\R^{n+1}$. This will play a crucial role in our proof of Theorem \ref{main}. It is also important to keep in mind that using \eqref{stoc}, and \cite[Propositions 2.3 and 2.4]{Gcm}, we infer that \eqref{Ga} defines a stochastically complete semigroup, and therefore for every $X\in \R^{n+1}_+$ and $t>0$ we have
\begin{equation}\label{Ga1}
\mathscr P^{(a)}_t 1(X) = \int_{\R^{n+1}_+} \G (X,X',t) (x'_{n+1})^a dX' = 1,
\end{equation}
and also \begin{equation}\label{st}\mathscr P^{(a)}_t \phi(X) \underset{t\to 0^+}{\longrightarrow} \phi(X).\end{equation}

We close this section by stating a preliminary  regularity result for \eqref{p0}  (we refer to \cite[Chap. 4]{Li} for the relevant notion of parabolic H\"older spaces). In its statement we denote by $\tilde U$ the even extension of $U$ across $\{x_{n+1}=0\}$, and for economy of notation we continue to denote with $\La$ the extension \eqref{po1} of the operator to the whole space-time space $\R^{n+1}\times \R$. We explicitly remark that the following Lemma \ref{reg} only requires that the matrix-valued function $A(x)$ be continuous.

\begin{lemma}\label{reg}
Let $U$ be a weak solution to \eqref{p0}. Then the function $\tilde U$ solves
\begin{equation}\label{po1}
\La \tilde U \overset{def}= \D(|x_{n+1}|^a A(x) \nabla \tilde U) - |x_{n+1}|^a \partial_t \tilde U=0
\end{equation}
in $\mathbb B_3 \times (-1, T]$, and moreover $\tilde U \in H^{1+\alpha}(\mathbb B_2 \times (0, T])$  for all $\alpha>0$. 
\end{lemma}

\begin{proof}
The fact that $\tilde U$ solves \eqref{po1} is standard. When $A(x) \equiv \mathbb I_{n+1}$,  it follows from \cite[Lemma A.1]{BDGP} that $ \tilde U \in H^{2+\beta}(\mathbb B_2 \times (0, T])$ for some $\beta >0$. Applying compactness arguments as in the proof of Proposition A.3 in \cite{BDGP}, we can conclude that  $\tilde U \in H^{1+\alpha}(\mathbb B_2 \times (0, T])$ for any $\alpha>0$. 

\end{proof}

\section{Proof of the main result}\label{s:m}
 
In this section we present the proof of Theorem \ref{main}. For notational convenience, throughout the proof we will respectively denote by $\tilde U, \tilde V$ the  symmetric extensions across the thin set $\{x_{n+1}=0\}$ of the relevant functions $U, V$. Also, in the ensuing computations we denote by $dS$ the surface measure on $\partial \mathbb B_2$. 

\begin{proof}[Proof of Theorem \ref{main}]
We consider the region $\mathbb B_2 \times (0,T]$. For a given $X \in \mathbb B_1$ and $\ve >0$, we  let
\begin{equation}\label{v}
V(X',t)= \G(X, X', T+\ve- t).
\end{equation}
Since the pole of $\G(X, X', T+\ve- t)$ is at the point $(X,T+\ve)$, the reflected function $\tilde V$ is smooth in $\mathbb B_2  \times (0, T)$, real-analytic in $X'$ for every fixed $t\in (0,T)$, and solves the following backward equation
\begin{equation}\label{back}
\La^\star \tilde V= \D(|x'_{n+1}|^a A(x') \n \tilde V)+ |x'_{n+1}|^a \partial_t \tilde V =0. 
\end{equation}
Let us notice that the equations \eqref{po1}  and \eqref{back}, satisfied by $\tilde U$ and $\tilde V$ respectively, give in the cylinder $\left[\mathbb B_2 \cap \{|x'_{n+1}| > \delta\}\right]  \times (0, T)$
\[
0 = \tilde V \La \tilde U - \tilde U \La^\star \tilde V = \D(|x'_{n+1}|^a(\tilde V A(x')\n \tilde U - \tilde U A(x')\n \tilde V) - |x'_{n+1}|^a(\tilde U \tilde V)_t.
\]
Integrating by parts on such set, using the regularity result in Lemma \ref{reg} and the zero Neumann condition \eqref{neu} satisfied by $\tilde V$, after letting $\delta \to 0$ we find
\begin{align}\label{i1o}
& 0= \int_{\mathbb B_2} \tilde V(X',T)\tilde U(X',T)|x'_{n+1}|^a dX' - \int_{\mathbb B_2} \tilde V(X',0) \tilde U(X',0) |x'_{n+1}|^a dX'
\\
& - \frac 12 \int_{0}^T dt \int_{\partial \mathbb B_2} \left\{\tilde V \langle A\nabla \tilde U, X'\rangle -  \tilde U \langle A\nabla \tilde V, X'\rangle\right\} |x'_{n+1}|^a dS(X').
\notag
\end{align}
If we now let $\ve \to 0^+$, using \eqref{st} we have
\begin{equation}\label{i2}
\int_{\mathbb B_2} \tilde V(X',T) \tilde U(X',T) |x'_{n+1}|^a dX\ \longrightarrow\  2 \tilde U(X,T).
\end{equation}
The factor $2$ in the right-hand side of \eqref{i2} is caused by the fact that, in view of \eqref{Ga}, \eqref{Ga1}, each half of the integral on $\mathbb B_2$ tends to $\tilde U(X,T)$ in the limit as $\ve\to 0^+$. Using \eqref{i2} in \eqref{i1o} we thus obtain
 \begin{align}\label{i4}
 & 2 \tilde U(X,T)= \int_{\mathbb B_2} \G(X,X',T) \tilde U(X', 0) |x'_{n+1}|^a dX' 
 \\
 & + \frac 12 \int_{0}^T dt \int_{\partial \mathbb B_2} \bigg\{\G(X,X',T-t) \langle A(x')\nabla \tilde U(X',t),X'\rangle
 \notag\\
 & -  \tilde U(X',t) \langle A(x')\nabla \G(X,X',T-t),X'\rangle\bigg\} |x'_{n+1}|^a dS(X').
 \notag
 \end{align} 
We note that \eqref{i4} constitutes a Green representation for $\tilde U(X,T)$. Since the functions $\G, \nabla \G$ are $C^\infty$ in $(X,T)$, \eqref{i4} and a fairly standard limiting argument imply, in particular, that a local solution to \eqref{p0} is $C^\infty$ in $(X,T)$. However, the real-analyticity in $X$ cannot similarly be obtained, and this is why we next resort to complexifying the function $\G$. In so doing, however, the estimates in the imaginary direction in $\mathbb C^{n+1}$ deteriorate and we need a more delicate analysis.  
 
As we have mentioned already, from Lemma \ref{ei},  the expression of $\G$ in \eqref{funda}, and from \eqref{besseries}, one sees that the symmetric extension of $\G(X,X',t)$ is  real-analytic in $X$ and $X'$ for $t<T$ and $|X-X'|<5$. Furthermore, a computation and the identity 
\begin{equation}\label{der}
\frac{d}{dw} [ w^{-\nu} I_\nu (w)] = w^{-\nu}I_{\nu +1} (w),
\end{equation} 
(see for instance \cite[(5.7.9) on p.110]{Le})
allow to verify that jointly in $X$ and $X'$ variables, 
\[
\langle A(x')\nabla \G(X,X',T-t),X'\rangle
\]
 is also a real-analytic function, symmetric in the variables $x_{n+1}$ and $x'_{n+1}$.
We now use the formula \eqref{i4} to extend $\ X\to \tilde U(X,T)$ to $Z\to \tilde U(Z,t)$, where $Z= X+ i Y$ ranges in an appropriate domain $D\subset \mathbb C^{n+1}$.  Using  Lemma \ref{ei}, \eqref{funda} and  \eqref{besseries}, it is seen that the first integral in the right-hand side of \eqref{i4} can be extended to an analytic function of $Z$ for $|Z| <2$. The analyticity of the second integral is not obvious and we thus proceed with the more delicate arguments that follow. For a given $\delta >0$, we let
 \begin{align}\label{i5}
F_\delta(Z) & = \int_{0}^{T-\delta} dt \int_{\partial \mathbb B_2} \bigg\{\G(Z,X',T-t) \langle A(x')\nabla \tilde U(X',t),X'\rangle \\
&-\tilde U(X',t) \langle A(x')\nabla \G(Z,X',T-t),X'\rangle\bigg\} |x'_{n+1}|^a dS(X')
\notag
\end{align}
From Lemma \ref{ei}, \eqref{funda} and \eqref{besseries} the function $F_{\delta}$ is analytic in the region $\{|Z|<3/2\}$.  
We intend to show that, for a suitably chosen number $\ve_0>0$, the holomorphic functions $F_{\delta}$'s are uniformly convergent  as $\delta \to 0$ in the region $$D(\ve_0) \overset{def}= \{Z= X + i Y\mid X \in \mathbb B_1\ \text{ and}\  |Y| \leq \ve_0\}.$$
Since uniformly convergent sequences of holomorphic functions have holomorphic limits (see e.g. \cite[Prop. 5 in Chap. 1]{Na}), we would infer that $\underset{\delta \to 0}{\lim}\ F_{\delta} = F$ is holomorphic in $D(\ve_0)$. On the other hand, we clearly have
\begin{align*}
F(Z) & = \int_{0}^{T} dt \int_{\partial \mathbb B_2} \bigg\{\G(Z,X',T-t) \langle A(x')\nabla \tilde U(X',t),X'\rangle \\
&-\tilde U(X',t) \langle A(x')\nabla \G(Z,X',T-t),X'\rangle\bigg\} |x'_{n+1}|^a dS(X').
\end{align*}
From  the representation \eqref{i4} we would thus conclude that $X\to \tilde U(X,T)$ is real-analytic in $\mathbb B_1$.

To prove the uniform convergence of the $F_{\delta}$'s in the appropriate region $D(\ve_0)\subset \mathbb C^{n+1}$, we proceed as follows.
We write $F_\delta(Z) = F^1_\delta(Z) - F^2_\delta(Z)$,
where
\begin{equation}\label{F1delta}
F^1_\delta(Z) = \int_{0}^{T} dt \int_{\partial \mathbb B_2} \G(Z,X',T-t) \langle A(x')\nabla \tilde U(X',t),X'\rangle |x'_{n+1}|^a dS(X'),
\end{equation}
\begin{equation}\label{F2delta}
F^2_\delta(Z) =  \int_{0}^{T} dt \int_{\partial \mathbb B_2} \tilde U(X',t) \langle A(x')\nabla \G(Z,X',T-t),X'\rangle |x'_{n+1}|^a dS(X'),
\end{equation}
 and prove that $F^k_{\delta}(Z)$, $k = 1, 2$ converge uniformly for $Z\in D(\ve_0)$. Since the arguments are essentially identical, we present details only for $F^1_\delta(Z)$, confining ourselves to briefly indicate at the end the changes necessary to treat $F^2_\delta(Z)$. We first note that \eqref{funda} and \eqref{fs} give 
 \begin{align}\label{exp}
 & \G(Z,X',T-t) = p(x + i y,x',T-t)\ p^{(a)}(x_{n+1}+i y_{n+1},x'_{n+1},T-t)
 \\
 &= p(x + i y,x',T-t) (2(T-t))^{-\frac{a+1}2} e^{-\frac{(x'_{n+1})^2 + x^2_{n+1} - y^2_{n+1} + 2 i x_{n+1} y_{n+1}}{4(T-t)}}
\notag\\
& \times\ \left(\frac{x'_{n+1} (x_{n+1}  + i y_{n+1})}{2(T-t)} \right)^{\frac{1-a}{2}} I_{\frac{a-1}{2}}\left(\frac{x'_{n+1} (x_{n+1}  + i y_{n+1})}{2(T-t)} \right).
\notag 
\end{align} 
To establish the uniform convergence of $F^1_\delta(Z)$ it will thus suffice to show that, for a sufficiently small choice of $\ve_0>0$, the function $\G(Z,X',T-t)$ is uniformly bounded as $t \to T$ for $Z \in D(\ve_0)$. For later use we notice that for $Z = X+i Y\in D(\ve_0)$ we have $|y_{n+1}|\le |Y| \leq \ve_0$. Since for $X'\in \mathbb B_2$ we have $|x'_{n+1}| \le |X'| \leq 2$, we thus have
\begin{equation}\label{set}
|x'_{n+1} y_{n+1}| \leq 2 \ve_0.
\end{equation}  
Furthermore, in view of  the symmetry in the $(n+1)$-th coordinate of $\G$ and $ \langle A(x')\nabla \G,X'\rangle$, it suffices to consider points $X' \in \partial \mathbb B_2 \cap \{x_{n+1}>0\}$, and $Z= X + i Y$ such that $X \in \overline{\mathbb B_1^+}$.  We split the analysis into Cases (1) \& (2), each of them composed of two subcases, (1a) \& (1b), and (2a) \& (2b). In the sequel $c_0, c_1$ and $C(a), c(a)$ will respectively denote the constants in \eqref{gest} of Lemma \ref{ei} and those in \eqref{bessel}. Also, from the definition \eqref{besseries} of the modified Bessel function it is clear that
\begin{align}\label{here}
&\bigg|\left(\frac{x'_{n+1} (x_{n+1}  + i y_{n+1})}{2(T-t)}\right)^{\frac{1-a}{2}} I_{\frac{a-1}{2}}\left(\frac{x'_{n+1} (x_{n+1}  + i y_{n+1})}{2(T-t)}\right)  \bigg|
\\
&\leq \bigg|\frac{x'_{n+1} (x_{n+1}  + i y_{n+1})}{2(T-t)}\bigg|^{\frac{1-a}{2}}  I_{\frac{a-1}{2}}\left(\left|\frac{x'_{n+1} (x_{n+1}  + i y_{n+1})}{2(T-t)}\right|\right).
\notag
\end{align}

\noindent \textbf{Case (1):} $(x'_{n+1})^2 + x_{n+1}^2 \leq   \ve_0$. If we assume that 
\begin{equation}\label{pi1}
\ve_0 < 1/16,
\end{equation}
 we see that it must be
\begin{equation}\label{lbd1}
|x_{n+1} - x'_{n+1}| \leq  |x_{n+1}| +|x'_{n+1}| \leq 2 \sqrt{\ve_0} <1/2.
\end{equation}
Moreover, since $X' \in \partial \mathbb B_2^+$ and $X  \in \mathbb B_1^+$, we have $|X'- X| >1$. Using this along with \eqref{lbd1}, by an application of  the triangle inequality we deduce that the following holds
\begin{equation}\label{lbd}
|x- x'| \geq  |X-X'|  - |x_{n+1} - x'_{n+1}| > 1/2.
\end{equation} 
We now distinguish two possibilities. 

\noindent \textbf{Case (1a):} $\bigg|\frac{x'_{n+1} (x_{n+1}  + i y_{n+1})}{2(T-t)}\bigg| \leq c(a)$. In such case, by \eqref{here} and the first inequality in \eqref{bessel}, we find  
\begin{align}\label{li1}
&\bigg|\left(\frac{x'_{n+1} (x_{n+1}  + i y_{n+1})}{2(T-t)}\right)^{\frac{1-a}{2}} I_{\frac{a-1}{2}}\left(\frac{x'_{n+1} (x_{n+1}  + i y_{n+1})}{2(T-t)}\right)  \bigg| \leq C(a).
\end{align}
If we now use in \eqref{exp} the first inequality in \eqref{gest}, \eqref{lbd},  \eqref{li1}, and also the fact that in $D(\ve_0)$ we have $|Y| \leq \ve_0$, we find that the following estimate holds
\begin{equation*}
|\G(Z,X',T-t)| \leq \frac{C(n,a)}{(T-t)^{(n+a+1)/2}} e^{-\frac{c_0/4 - (c_1+1) \ve_0^2}{4(T-t)}}.
\end{equation*}
If in addition to \eqref{pi1} we assume that 
\begin{equation}\label{iu2}
\ve_0^2 < \frac{c_0}{8(c_1+1)},
\end{equation}
we can thus guarantee that $\G(Z,X',T-t)$ is uniformly bounded as $t \to T$.

\noindent \textbf{Case (1b):} $\bigg|\frac{x'_{n+1} (x_{n+1}  + i y_{n+1})}{2(T-t)}\bigg| > c(a)$. 
 In this case, \eqref{here} and the second inequality in \eqref{bessel} imply
 \begin{align}\label{l0}
 & \left|\left(\frac{x'_{n+1} (x_{n+1}  + i y_{n+1})}{2(T-t)}\right)^{\frac{1-a}{2}} I_{\frac{a-1}{2}}\left(\frac{x'_{n+1} (x_{n+1}  + i y_{n+1})}{2(T-t)}\right)  \right| 
\\
& \le C(a) c(a)^{-\frac 12}\bigg|\frac{x'_{n+1} (x_{n+1}  + i y_{n+1})}{2(T-t)}\bigg|^{\frac{1-a}2} \exp\left\{{\bigg|\frac{x'_{n+1} (x_{n+1}  + i y_{n+1})}{2(T-t)}\bigg|}\right\}
\notag\\
& \leq \overline C( a) (T-t)^{-\frac{1-a}2} \exp\left\{{\frac{4\ve_0 + 2|x'_{n+1}| |x_{n+1}|}{4(T-t)}}\right\},
\notag 
\end{align} 
where we have used  the inequality \eqref{set}. If we now use in \eqref{exp} the first estimate in \eqref{gest}, 
\eqref{lbd}, \eqref{l0} and again $|Y| \leq \ve_0$, we find 
\begin{align*}
|\G(Z,X',T-t)| & \leq \frac{C(n, a)}{(T-t)^{\frac{n+2}2}} e^{-\frac{c_0/4 + |x'_{n+1} - x_{n+1}|^2  - 4 \ve_0 - (c_1+1)\ve_0^2 }{4(T-t)}} 
\\
& \leq  \frac{C(n,a)}{(T-t)^{\frac{n+2}2}}  e^{-\frac{c_0}{32(T-t)}},\notag\end{align*}
provided that
\begin{equation}\label{iu4} 
4 \ve_0 + (c_1+1) \ve_0^2 < \frac{c_0}8.
\end{equation}
Under the hypothesis \eqref{pi1} and \eqref{iu4}, we again obtain the uniform boundedness of $\G(Z,X',T-t)$ as $t\to T$.
We next consider the situation complementary to Case (1).

\noindent \textbf{Case (2):} $(x'_{n+1})^2 + x_{n+1}^2 > \ve_0$. In this case also we distinguish two possibilities, Cases (2a) and (2b). Case (2a) will be further subdivided into two subcases, Case $(2a)_1$ and $(2a)_2$.
 
\noindent \textbf{Case (2a):} $\frac{x'_{n+1} x_{n+1}}{2(T-t)}\leq c(a)$. As we have said, this case is further subdivided into two subcases. 
 
\noindent \textbf{Case (2a)$_1$ :} $\bigg|\frac{x'_{n+1} (x_{n+1}  + i y_{n+1})}{2(T-t)}\bigg| \leq c(a)$. As in Case (1a), we again have the bound \eqref{li1}, which we use in \eqref{exp} along with the first estimate in \eqref{gest}, the present assumptions that $(x'_{n+1})^2 + x_{n+1}^2 > \ve_0$, and $|Y|\leq \ve_0$, to obtain (note that we have used $|\exp{\big(-\frac{2 i x_{n+1} y_{n+1}}{4(T-t)}\big)}| = 1$)
 \[
 |\G(Z,X',T-t)| \leq \frac{C(n,a)}{(T-t)^{\frac{n+a+1}2}} e^{- \frac{\ve_0 - (c_1 +1) \ve_0^2}{4(T-t)}}  \leq \frac{C(n,a)}{(T-t)^{\frac{n+a+1}2}} e^{- \frac{\ve_0}{8(T-t)}},
\]
 provided $\ve_0$ is such that
\begin{equation}\label{o1}
1 > 2(c_1 +1) \ve_0.
\end{equation}  
Under \eqref{o1} we thus have that $\G(Z, X',T-t)$ is uniformly bounded as $t \to T$.

\noindent \textbf{Case (2a)$_2$:} $\bigg|\frac{x'_{n+1} (x_{n+1}  + iy_{n+1})}{2(T-t)}\bigg| > c(a)$. Similarly to Case (1b), we again have the  inequality \eqref{l0}, which we use in \eqref{exp}, along with the first estimate in \eqref{gest}, and the already observed inequality $|X'- X| >1$, to find
 \begin{align*}
 |\G(Z,X',T-t)| & \leq \frac{C(n, a)}{(T-t)^{(n+2)/2}} e^{-\frac{c_0 |x- x'|^2 +|x_{n+1} - x'_{n+1}|^2 -  4\ve_0 -  (c_1 +1)\ve_0^2}{4(T-t)}}\\
& \leq  \frac{C(n, a)}{(T-t)^{(n+2)/2}} e^{-\frac{\operatorname{min}(c_0, 1) |X-X'|^2 -  4\ve_0 -  (c_1 +1)\ve_0^2}{4(T-t)}}\notag\\ 
&\leq  \frac{C(n, a)}{(T-t)^{(n+2)/2}} e^{-\frac{\operatorname{min}(c_0, 1)  -  4\ve_0 -  (c_1 +1)\ve_0^2}{4(T-t)}}  \notag   \\& \leq  \frac{C(n, a)}{(T-t)^{(n+2)/2}}  e^{-\frac{7/8 \operatorname{min}(c_0,1)}{4(T-t)}}\  ,\notag\end{align*}  
provided  $\ve_0$ satisfies
\begin{equation}\label{iu5}
\frac{\operatorname{min}(c_0, 1)}{8} >    4\ve_0  + (c_1 +1)\ve_0^2.
\end{equation}  
Thus, also in this Case (2a)$_2$ we find that $\G(Z,X', T-t)$ is uniformly bounded as $t \to T$. Combining this with the discussion in Case (2a)$_1$, we infer that the same conclusion holds in Case (2a) provided that $\ve_0$ is small enough.

 \noindent \textbf{Case (2b)} $\frac{x'_{n+1} x_{n+1}}{2(T-t)}> c(a)$. Since this assumption obviously implies  $\bigg|\frac{x'_{n+1} (x_{n+1}  + iy_{n+1})}{2(T-t)}\bigg| > c(a)$, we can repeat the arguments in Case (2a)$_2$ and obtain a uniform bound  for $\G(Z, X,T-t)$ as $t \to T$.

With all this being said, we now choose $\ve_0>0$ so small that \eqref{pi1}, \eqref{iu2}, \eqref{iu4}, \eqref{o1} and \eqref{iu5} concurrently hold.  This guarantees that the function $F_\delta^1(Z)$ in \eqref{F1delta} converges uniformly for $Z\in D(\ve_0)$ as $\delta \to 0$.
For the uniform of convergence of $F^2_\delta(Z)$, 
we use \eqref{der} to obtain a representation of 
\[
\langle A(x')\nabla \G(Z,X',T-t), X'\rangle
\]
similar to that in \eqref{exp} for $\G(Z,X',T-t)$, but this time in terms of the modified Bessel functions $I_{\frac{a+1}2}$ and $I_{\frac{a-1}2}$. Once this is observed, we can repeat the arguments in Case (1) and (2) above, except that we now need to also use the second estimate in \eqref{gest}. Since at this point the reader can easily fill in the necessary details, we skip them altogether and just affirm that $F_{\delta}(Z)$ converges uniformly in $Z\in D(\ve_0)$. In view of the discussion after \eqref{i5}, we conclude that $X \to \tilde U(X,T)$  is real-analytic in $\mathbb B_1$ and this finishes the proof of the theorem.  

\end{proof}

With Theorem \ref{main} in hand, we now provide the

\begin{proof}[Proof of Proposition \ref{wucp}]
Given $u$ as in the statement of the proposition, let $U$ be the solution of the corresponding extension problem \eqref{p0}, i.e., 
\begin{equation}
\begin{cases}
\La U=0\ \text{in $\{x_{n+1}>0\}$}
\\
U=\py U=0\ \text{on}\ \{x_{n+1}=0\} \cap [B_1 \times (-1, 0)],
\end{cases}
\end{equation}
with $\La$ as in \eqref{p0} and where the $(n+1) \times (n+1)$ matrix  valued function $A(x)$ is of the form \eqref{type}. From Theorem \ref{main} we infer that the evenly reflected $\tilde U$ is space-like real analytic in $\mathbb B_1 \times (-1, 0)$.
Moreover, since $U$ and $\py U$ both vanish on $\{x_{n+1}=0\}$, by an argument in  \cite[Lemma 5.1]{LLR} (see also the  proof of \cite[Lemma 7.7]{BG}), which involves repeated differentiation in $y$-variable and a bootstrap type argument, it follows that for every $t \in (-1, 0)$ the function $\tilde U(\cdot,t)$ vanishes to infinite order  in the $x_{n+1}$ variable  at every $(x_0, 0) \in \mathbb B_1 \cap \{x_{n+1}=0\}$ \footnote{It is worth mentioning here that, although in the cited works \cite{BG} and \cite{LLR} only the case $B= \mathbb I_n$ was treated, the above mentioned bootstrap argument works unchanged for a smooth $B(x)$.}. In view of the real-analyticity of $\tilde U(\cdot, t)$ in $\mathbb B_1 $ we conclude that  $\tilde U(\cdot, t) \equiv 0$  in $\mathbb B_1$ for every $t \in (-1, 0)$.  We now note that  away from $\{x_{n+1}=0\}$,  $\tilde U$ solves a uniformly parabolic PDE with smooth
coefficients and vanishes identically in the  $\mathbb B_1^+ \times (-1, 0)$. We can thus appeal to \cite[Theor. 1]{AV} to
assert that $\tilde U$ vanishes to infinite order both in space and time  at every $(X, t)
 \in \mathbb B_1^+ \times  (-1,0)$. At this point, we can use the strong unique continuation result in \cite[Theor. 1]{EF} to finally
conclude that $U(X, t)= 0$ for $(X,t) \in \R^{n+1}_+ \times (-1,0)$. Letting $x_{n+1}=0$, this implies $u(x,t) = U((x,0),t) \equiv 0$ for $(x,t)\in \R^n \times (-1, 0)$.   This completes the proof of the proposition.

\end{proof}

\end{document}